%% file: main.tex
\documentclass[11pt]{article}

\input{preamble}

\title{A Counterexample to the Exact Eight-Point Angular-Energy Conjecture}
\author{Jinze Zhao\\University of California, San Diego\\\texttt{jiz419@ucsd.edu}}
\date{}

\begin{document}
\maketitle

\begin{abstract}
\input{sections/abstract}
\end{abstract}


\input{sections/introduction}
\input{sections/main}
\input{sections/conclusion}
\input{sections/disclosure}
\bibliographystyle{plainnat}
\bibliography{references}

\appendix
\input{sections/appendix}
\end{document}

%% file: preamble.tex
\usepackage[T1]{fontenc}
\usepackage[utf8]{inputenc}
\usepackage{lmodern}
\usepackage[margin=1in]{geometry}
\usepackage{microtype}
\usepackage{amsmath,amssymb,amsthm,mathtools}
\usepackage{bm}
\usepackage{graphicx}
\graphicspath{{figures/}}
\usepackage{booktabs}
\usepackage{array}
\usepackage[shortlabels]{enumitem}
\usepackage[numbers,sort&compress]{natbib}
\usepackage{xcolor}
\usepackage[colorlinks=true,linkcolor=blue,citecolor=blue,urlcolor=blue]{hyperref}
\usepackage[capitalise,noabbrev]{cleveref}

\theoremstyle{plain}
\newtheorem{theorem}{Theorem}[section]
\newtheorem{lemma}[theorem]{Lemma}
\newtheorem{proposition}[theorem]{Proposition}

\theoremstyle{definition}

\newtheorem{conjecture}[theorem]{Conjecture}
\theoremstyle{remark}

\newcommand{\R}{\mathbb{R}}

\newcommand{\norm}[1]{\left\lVert #1\right\rVert}
\newcommand{\ip}[2]{\left\langle #1,#2\right\rangle}

\newcommand{\hzero}{h_0}
\newcommand{\hstar}{h_*}
\newcommand{\Q}{\mathcal{Q}}
\newcommand{\A}{\mathcal{A}}
\newcommand{\F}{\mathcal{F}}
\hypersetup{pdftitle={A Counterexample to the Exact Eight-Point Angular-Energy Conjecture}}

%% file: sections/abstract.tex
We give an exact counterexample to the eight-point angular-energy conjecture stated by Mouchet in 2005 and published in 2006. The conjecture specifies both a radical value for the global maximum and a square antiprism with unit-side horizontal squares separated by $\hzero=\sqrt{1+2\sqrt{2}}/2$. We prove that the same family at the rational separation $489/500$ has a strictly larger angular sum. The proof reduces the sum of all $168$ angle cosines to a one-variable formula, factors its derivative, and certifies its sign using rational inequalities. We also characterize the unique maximizing separation within this family. All comparisons needed for the counterexample are exact. The result disproves the exact stated conjecture; it does not determine the unrestricted eight-point maximum. A targeted literature search located no previous correction, but priority has not been established exhaustively.

%% file: sections/introduction.tex
\section{Introduction and precise scope}
\label{sec:introduction}

Consider distinct points $x_1,\ldots,x_N$ in $\R^3$. For each triangle formed by three of the points, sum the cosines of its three interior angles, and then sum over all triangles. Maximizing this quantity is a finite-dimensional, nonconvex geometric optimization problem. Its formulation in terms of actual point coordinates enforces geometric compatibility among the angles.

Mouchet introduced explicit conjectures for this objective in connection with bounds for attractive Coulombian systems \citep{Mouchet2006JMP}. The eight-point claim appears as Conjecture~3, labelled $(C_8)$, in the May~25, 2005 preprint and the 2006 journal article. A companion paper discusses the underlying optimization \citep{Mouchet2006NPA}. Hiriart-Urruty subsequently included the eight-point question, with approximate numerical parameters, in Problem~4 of a collection of optimization and matrix-analysis problems \citep{HiriartUrruty2009}.

The target here is the \emph{exact} eight-point claim: a specified separation of two rotated squares and a specified radical value for the maximum. A small improvement that is invisible at the reported three-decimal precision suffices to refute that claim. We exhibit such an improvement and prove it without floating-point assumptions.

\subsection{The configuration and the conjectured value}
\label{subsec:statement}

Let $\Q_N$ denote the set of ordered $N$-tuples of pairwise distinct points of $\R^3$. Define
\begin{equation}
 F_N(x_1,\ldots,x_N)
 =\sum_{i=1}^N\ \sum_{\substack{j<k\\j,k\ne i}}
 \frac{\ip{x_j-x_i}{x_k-x_i}}
 {\norm{x_j-x_i}\,\norm{x_k-x_i}}.
 \label{eq:definition}
\end{equation}
Every unordered triple contributes exactly three terms, one at each vertex. In particular, $F_8$ contains $3\binom{8}{3}=168$ terms. The formula is also well defined for collinear triples: the participating vectors remain nonzero. Our counterexample has no coincident points and requires no limiting configuration.

Write $s=\sqrt{2}$ and $R=1/\sqrt{2}$. For $h\geq 0$, define eight points by
\begin{align}
 p_j(h)&=\left(R\cos\frac{j\pi}{2},\ R\sin\frac{j\pi}{2},\ -\frac{h}{2}\right),
 \label{eq:lower-square}\\
 q_j(h)&=\left(R\cos\left(\frac{\pi}{4}+\frac{j\pi}{2}\right),\
 R\sin\left(\frac{\pi}{4}+\frac{j\pi}{2}\right),\ \frac{h}{2}\right),
 \qquad j=0,1,2,3.
 \label{eq:upper-square}
\end{align}
Each square has side length $R\sqrt{2}=1$. The planes are separated by $h$, and the squares are rotated by $\pi/4$. Denote this ordered configuration by $\A(h)$, and put
\begin{equation}
 \F(h)=F_8(\A(h)),\qquad
 \hzero=\frac{\sqrt{1+2\sqrt{2}}}{2},\qquad h_1=\frac{489}{500}.
 \label{eq:heights}
\end{equation}
The points remain distinct at $h=0$, where their projections form a regular octagon.

\begin{conjecture}[Exact eight-point claim, Mouchet $(C_8)$]
\label{conj:c8}
The supremum of $F_8$ over $\Q_8$ equals
\begin{equation}
\begin{split}
 C_0=16\left[
 \frac45+\frac1{\sqrt2}+\frac1{\sqrt5}
 +\frac{4(1+\sqrt2)}{\sqrt5\sqrt{5+4\sqrt2}}
 \right.\\[-2pt]
 \left.
 {}+\frac{3+2\sqrt2}{\sqrt{5+4\sqrt2}}
 -\frac1{5+4\sqrt2}
 \right],
\end{split}
\label{eq:conjectured-value}
\end{equation}
and it is attained at $\A(\hzero)$.
\end{conjecture}

The mathematical content of this statement is reproduced from \citet[Conjecture~3]{Mouchet2006JMP}; the notation above is ours. We will verify directly that $\F(\hzero)=C_0$ so that the comparison addresses both parts of the claim.

\begin{theorem}[Explicit counterexample]
\label{thm:counterexample}
The configurations in \cref{eq:lower-square,eq:upper-square} satisfy
\begin{equation}
 F_8\!\left(\A\!\left(\frac{489}{500}\right)\right)
 > F_8(\A(\hzero))=C_0.
 \label{eq:counterexample}
\end{equation}
Moreover, $\F'(\hzero)<0$. Thus $\A(\hzero)$ is not a local maximizer, even within the one-parameter family $\A(h)$, and Conjecture 1.1 is false.
\end{theorem}




%% file: sections/main.tex
\section{Reducing the angular sum to a scalar function}
\label{sec:reduction}

\subsection{An identity for arbitrary point configurations}

\begin{lemma}[Vertex-sum identity]
\label{lem:vertex}
For $(x_1,\ldots,x_N)\in\Q_N$, define
\begin{equation}
 u_{ij}=\frac{x_j-x_i}{\norm{x_j-x_i}}
 \quad(j\ne i),\qquad
 v_i=\sum_{j\ne i}u_{ij}.
 \label{eq:unit-vectors}
\end{equation}
Then
\begin{equation}
 F_N=\frac12\sum_{i=1}^N\norm{v_i}^2-\frac{N(N-1)}2.
 \label{eq:vertex-sum}
\end{equation}
\end{lemma}
\begin{proof}
Fix a vertex $i$. Since every $u_{ij}$ has norm one, expansion of the squared norm gives
\begin{align*}
 \norm{v_i}^2
 &=\ip{\sum_{j\ne i}u_{ij}}{\sum_{k\ne i}u_{ik}}\\
 &=\sum_{j\ne i}\norm{u_{ij}}^2
   +2\sum_{\substack{j<k\\j,k\ne i}}\ip{u_{ij}}{u_{ik}}\\
 &=N-1+2\sum_{\substack{j<k\\j,k\ne i}}
 \frac{\ip{x_j-x_i}{x_k-x_i}}
 {\norm{x_j-x_i}\,\norm{x_k-x_i}}.
\end{align*}
Subtract $N-1$, divide by two, and sum over $i=1,\ldots,N$. The double sum is exactly \cref{eq:definition}, and the constants sum to $N(N-1)/2$. This proves \cref{eq:vertex-sum}.
\end{proof}

\subsection{The seven directions from a vertex of the antiprism}

Set
\begin{equation}
 \alpha=1-\frac{s}{2},\qquad
 \beta=1+\frac{s}{2},\qquad
 d_-(h)=\sqrt{h^2+\alpha},\qquad
 d_+(h)=\sqrt{h^2+\beta}.
 \label{eq:distances}
\end{equation}
Because $1<s<2$, both $\alpha$ and $\beta$ are positive. Hence both distances are positive for every $h\geq0$.

\begin{lemma}[Exact one-variable expression]
\label{lem:scalar}
For $h\geq0$, put
\begin{align}
 a(h)&=1+s+\frac{s-1}{d_-(h)}+\frac{s+1}{d_+(h)},
 \label{eq:a}\\
 b(h)&=2h\left(\frac1{d_-(h)}+\frac1{d_+(h)}\right).
 \label{eq:b}
\end{align}
Then
\begin{equation}
 \F(h)=4\bigl(a(h)^2+b(h)^2-7\bigr).
 \label{eq:scalar}
\end{equation}
In particular, $\F$ is continuously differentiable on $[0,\infty)$ and smooth on $(0,\infty)$.
\end{lemma}
\begin{proof}
Take the lower vertex $p_0=(R,0,-h/2)$. The other three lower vertices have displacement vectors
\[
 (-R,R,0),\qquad (-2R,0,0),\qquad (-R,-R,0).
\]
Their lengths are $1,\sqrt2,1$, respectively. Their unit vectors therefore sum to
\[
 (-R,R,0)+(-1,0,0)+(-R,-R,0)=(-(1+s),0,0).
\]

The four upper vertices have horizontal coordinates
\[
 (1/2,1/2),\quad(-1/2,1/2),\quad
 (-1/2,-1/2),\quad(1/2,-1/2).
\]
For the two upper vertices with first coordinate $1/2$, the squared horizontal distance from $(R,0)$ is
\[
 (R-1/2)^2+(1/2)^2
 =R^2-R+1/2=1-\frac{s}{2}=\alpha.
\]
For the two with first coordinate $-1/2$, it is
\[
 (R+1/2)^2+(1/2)^2
 =R^2+R+1/2=1+\frac{s}{2}=\beta.
\]
Their full distances are consequently $d_-$ and $d_+$, respectively. The two vectors of length $d_-$ have cancelling second components; their summed unit-vector contribution is
\[
 \left(\frac{1-2R}{d_-},0,\frac{2h}{d_-}\right)
 =\left(-\frac{s-1}{d_-},0,\frac{2h}{d_-}\right).
\]
The contribution of the other pair is
\[
 \left(\frac{-1-2R}{d_+},0,\frac{2h}{d_+}\right)
 =\left(-\frac{s+1}{d_+},0,\frac{2h}{d_+}\right).
\]
Adding all seven directions gives $v_{p_0}=(-a(h),0,b(h))$.

Rotation by $\pi/2$ around the vertical axis permutes the vertices in each layer. Rotation by $\pi/4$ combined with reflection in the horizontal plane exchanges the layers and permutes all eight vertices. These isometries act transitively on the vertex set and preserve the lengths of the corresponding direction sums. Thus every $v_i$ has squared norm $a(h)^2+b(h)^2$. Applying \cref{lem:vertex} with $N=8$ yields
\[
 \F(h)=\frac12\,8\bigl(a(h)^2+b(h)^2\bigr)-\frac{8\cdot7}{2},
\]
which is \cref{eq:scalar}. The differentiability assertion follows because the radicands in \cref{eq:distances} are strictly positive.
\end{proof}

\subsection{Agreement with the printed radical value}

\begin{lemma}
\label{lem:printed-value}
The value in \cref{eq:conjectured-value} equals $\F(\hzero)$.
\end{lemma}
\begin{proof}
For readability, write $d_-=d_-(h)$ and $d_+=d_+(h)$. Squaring \cref{eq:a,eq:b} and collecting terms gives
\begin{align}
 a^2+b^2-7
 ={}&(1+s)^2-7
 +\frac{2(1+s)(s-1)}{d_-}
 +\frac{2(1+s)(s+1)}{d_+}\notag\\
 &+\frac{(s-1)^2+4h^2}{d_-^2}
 +\frac{(s+1)^2+4h^2}{d_+^2}
 +\frac{2(s-1)(s+1)+8h^2}{d_-d_+}.
 \label{eq:expansion}
\end{align}
Use $s^2=2$ to obtain
\begin{gather*}
 (1+s)^2=3+2s,\qquad (1+s)(s-1)=1,\qquad
 (1+s)(s+1)=3+2s,\\
 (s-1)^2+4h^2=4d_-^2-1,\qquad
 (s+1)^2+4h^2=4d_+^2-1,\qquad
 (s-1)(s+1)=1.
\end{gather*}
Substitution in \cref{eq:expansion} therefore yields
\begin{equation}
\begin{split}
 a^2+b^2-7={}&4+2s+\frac2{d_-}+\frac{2(3+2s)}{d_+}
 -\frac1{d_-^2}-\frac1{d_+^2}
 +\frac{2+8h^2}{d_-d_+}.
\end{split}
\label{eq:expanded-simplified}
\end{equation}
At $h=\hzero$, we have $h^2=(1+2s)/4$, so
\[
 d_-^2=\frac54,\qquad d_+^2=\frac{5+4s}{4},\qquad
 2+8h^2=4(1+s).
\]
Let $u=\sqrt5$ and $t=\sqrt{5+4s}$. Then $d_-=u/2$, $d_+=t/2$. Dividing \cref{eq:expanded-simplified} by four and using \cref{eq:scalar}, we obtain
\begin{align*}
 \frac{\F(\hzero)}{16}
 &=1+\frac{s}{2}+\frac1u+\frac{3+2s}{t}
   -\frac15-\frac1{t^2}+\frac{4(1+s)}{ut}\\
 &=\frac45+\frac1{\sqrt2}+\frac1{\sqrt5}
 +\frac{4(1+\sqrt2)}{\sqrt5\sqrt{5+4\sqrt2}}
 +\frac{3+2\sqrt2}{\sqrt{5+4\sqrt2}}
 -\frac1{5+4\sqrt2}.
\end{align*}
Here $s/2=1/\sqrt2$ because $s^2=2$. The last expression is $C_0/16$.
\end{proof}

\section{Derivative factorization and the unique best height}
\label{sec:derivative}

Define
\begin{equation}
 c=3-2s=(s-1)^2>0,\qquad
 K(h)=\frac1{d_-(h)}+\frac{c}{d_+(h)}-1,
 \label{eq:K}
\end{equation}
and
\begin{equation}
 T(h)=\frac{\alpha}{d_-(h)^3}+\frac{\beta}{d_+(h)^3}>0.
 \label{eq:T}
\end{equation}

\begin{lemma}[Sign of the derivative]
\label{lem:derivative}
For every $h>0$,
\begin{equation}
 \F'(h)=8h(2+s)T(h)K(h).
 \label{eq:factor}
\end{equation}
Thus $\F'(h)$ has the same sign as $K(h)$.
\end{lemma}
\begin{proof}
Differentiating $d_\pm(h)^2=h^2+\alpha$ or $h^2+\beta$, respectively, gives
\[
 d_\pm'(h)=\frac{h}{d_\pm(h)},\qquad
 \left(\frac1{d_\pm(h)}\right)'=-\frac{h}{d_\pm(h)^3}.
\]
The relations $s\alpha=s-1$ and $s\beta=s+1$ imply
\begin{equation}
 a'(h)=-h\left(\frac{s-1}{d_-^3}+\frac{s+1}{d_+^3}\right)
 =-shT(h).
 \label{eq:a-prime}
\end{equation}
If $S(h)=1/d_-+1/d_+$, differentiation of $b(h)=2hS(h)$ gives
\begin{align}
 b'(h)
 &=2\left(\frac1{d_-}+\frac1{d_+}\right)
   -2h^2\left(\frac1{d_-^3}+\frac1{d_+^3}\right)\notag\\
 &=2\left(\frac{d_-^2-h^2}{d_-^3}
          +\frac{d_+^2-h^2}{d_+^3}\right)
 =2T(h).
 \label{eq:b-prime}
\end{align}
By \cref{eq:scalar,eq:a-prime,eq:b-prime},
\begin{align*}
 \F'(h)&=8\bigl(a(h)a'(h)+b(h)b'(h)\bigr)\\
 &=8T(h)\bigl(-sh\,a(h)+2b(h)\bigr)\\
 &=8hT(h)\bigl(4S(h)-s\,a(h)\bigr).
\end{align*}
Now expand the final bracket:
\begin{align*}
 4S-sa
 &=\frac{4-s(s-1)}{d_-}+\frac{4-s(s+1)}{d_+}-s(1+s)\\
 &=\frac{2+s}{d_-}+\frac{2-s}{d_+}-(2+s)\\
 &=(2+s)\left(\frac1{d_-}+\frac{3-2s}{d_+}-1\right).
\end{align*}
The last equality uses
\[
 \frac{2-s}{2+s}=\frac{(2-s)^2}{4-s^2}
 =\frac{6-4s}{2}=3-2s.
\]
This proves \cref{eq:factor}. All factors preceding $K(h)$ are strictly positive for $h>0$.
\end{proof}

\begin{theorem}[Complete optimization within the family]
\label{thm:family}
There exists exactly one $\hstar>0$ satisfying
\begin{equation}
 \frac1{\sqrt{\hstar^2+1-1/\sqrt2}}
 +\frac{3-2\sqrt2}{\sqrt{\hstar^2+1+1/\sqrt2}}=1.
 \label{eq:root}
\end{equation}
The function $\F$ is strictly increasing on $[0,\hstar]$ and strictly decreasing on $[\hstar,\infty)$. Consequently $\A(\hstar)$ is the unique maximizer in the family when parametrized by $h\geq0$.
\end{theorem}
\begin{proof}
The function $K$ is continuous on $[0,\infty)$. For $h>0$ its derivative is
\[
 K'(h)=-\frac{h}{d_-(h)^3}-\frac{ch}{d_+(h)^3}<0,
\]
because $c>0$ and $d_\pm>0$. Hence $K$ is strictly decreasing. Also $0<\alpha<1$, so
\[
 K(0)=\frac1{\sqrt\alpha}+\frac{c}{\sqrt\beta}-1>0.
\]
As $h\to\infty$, both reciprocals in \cref{eq:K} tend to zero; therefore $K(h)\to-1$. The intermediate value theorem supplies a positive zero $\hstar$, and strict decrease makes it unique. Equation~\eqref{eq:root} is precisely $K(\hstar)=0$, since $s/2=1/\sqrt2$.

For $0<h<\hstar$ we have $K(h)>0$, and for $h>\hstar$ we have $K(h)<0$. By \cref{lem:derivative}, the same signs hold for $\F'(h)$. To make strict monotonicity explicit, take $0\leq u<v\leq\hstar$. The mean value theorem gives
\[
 \F(v)-\F(u)=\F'(\xi)(v-u)
\]
for some $\xi\in(u,v)\subset(0,\hstar)$. The right-hand side is positive. If $\hstar\leq u<v$, the same argument has $\xi>\hstar$ and makes the difference negative. Thus every height different from $\hstar$ has a strictly smaller value.
\end{proof}

\section{An exact rational certificate for the counterexample}
\label{sec:certificate}

It remains to locate $h_1=489/500$ on the decreasing side of $\F$, and to compare it with $\hzero$. The following certificate uses only rational comparisons and positivity of square roots.

\begin{lemma}[Rational bounds at the explicit height]
\label{lem:rational}
The heights satisfy
\begin{equation}
 0<\hstar<h_1<\hzero.
 \label{eq:height-order}
\end{equation}
More specifically,
\begin{equation}
 K(h_1)<-\frac{2409}{15200720}<0.
 \label{eq:K-negative}
\end{equation}
\end{lemma}
\begin{proof}
We first establish the elementary bounds
\begin{equation}
 \frac{7071}{5000}<s<\frac{99}{70}.
 \label{eq:sqrt-bounds}
\end{equation}
All three quantities are positive. Squaring the proposed lower bound gives
\[
 \left(\frac{7071}{5000}\right)^2
 =\frac{49999041}{25000000}
 <\frac{50000000}{25000000}=2,
\]
whereas squaring the upper bound gives
\[
 \left(\frac{99}{70}\right)^2
 =\frac{9801}{4900}>\frac{9800}{4900}=2.
\]
Since the square function is strictly increasing on nonnegative numbers, \cref{eq:sqrt-bounds} follows.

We have $h_1^2=239121/250000$. Using the upper bound on $s$ in the expression for $d_-^2$ gives
\[
 d_-(h_1)^2>\frac{239121}{250000}+1-\frac{99}{140}.
\]
The rational difference
\begin{equation}
 \frac{239121}{250000}+1-\frac{99}{140}
 -\left(\frac{11177}{10000}\right)^2
 =\frac{61497}{700000000}>0
 \label{eq:dminus-certificate}
\end{equation}
therefore implies
\begin{equation}
 d_-(h_1)>\frac{11177}{10000}.
 \label{eq:dminus-bound}
\end{equation}
For the other distance, use the lower bound on $s$:
\[
 d_+(h_1)^2>\frac{239121}{250000}+1+\frac{7071}{10000}.
\]
Here
\begin{equation}
 \frac{239121}{250000}+1+\frac{7071}{10000}
 -\left(\frac{204}{125}\right)^2
 =\frac1{6250}>0,
 \label{eq:dplus-certificate}
\end{equation}
so
\begin{equation}
 d_+(h_1)>\frac{204}{125}.
 \label{eq:dplus-bound}
\end{equation}
Again the passage from squared inequalities to these inequalities is valid because every quantity is positive.

The coefficient $c=(s-1)^2$ is positive, and the lower bound on $s$ gives
\begin{equation}
 0<c=3-2s<3-2\frac{7071}{5000}=\frac{429}{2500}.
 \label{eq:c-bound}
\end{equation}
Inverting \cref{eq:dminus-bound,eq:dplus-bound}, then using \cref{eq:c-bound}, yields
\begin{align*}
 K(h_1)+1
 &=\frac1{d_-(h_1)}+\frac{c}{d_+(h_1)}\\
 &<\frac{10000}{11177}+\frac{429/2500}{204/125}\\
 &=\frac{10000}{11177}+\frac{143}{1360}\\
 &=\frac{15198311}{15200720}
 =1-\frac{2409}{15200720}<1.
\end{align*}
This proves \cref{eq:K-negative}. Since $K$ is strictly decreasing and $K(\hstar)=0$, it follows that $\hstar<h_1$.

Finally, \cref{eq:sqrt-bounds} also gives
\begin{align*}
 \hzero^2-h_1^2
 &=\frac{1+2s}{4}-\frac{239121}{250000}\\
 &>\frac{1+2(7071/5000)}4-\frac{239121}{250000}\\
 &=\frac{77}{125000}>0.
\end{align*}
Both heights are positive, so $h_1<\hzero$. Combining the inequalities proves \cref{eq:height-order}.
\end{proof}

\begin{proof}[Proof of \cref{thm:counterexample}]
By \cref{lem:rational}, the interval $[h_1,\hzero]$ lies strictly to the right of $\hstar$. By \cref{thm:family}, $\F$ is strictly decreasing there. Consequently
\[
 \F(h_1)>\F(\hzero).
\]
Lemma~\ref{lem:printed-value} identifies the latter value with $C_0$, proving \cref{eq:counterexample}. Moreover $K(\hzero)<K(h_1)<0$, and \cref{lem:derivative} implies $\F'(\hzero)<0$.

For every sufficiently small $\varepsilon>0$ with $\hzero-\varepsilon>h_1$, strict decrease gives
\[
 \F(\hzero-\varepsilon)>\F(\hzero).
\]
The coordinates of $\A(\hzero-\varepsilon)$ tend to those of $\A(\hzero)$ as $\varepsilon\to0$. Hence every neighborhood of the conjectured configuration contains a feasible configuration with larger objective. It is not a local maximum in $\Q_8$, and in particular cannot attain the asserted global maximum. Since $\A(h_1)\in\Q_8$, we also have $\sup_{\Q_8}F_8\geq\F(h_1)>C_0$.
\end{proof}

\section{Reproducibility and numerical size}
\label{sec:reproducibility}

The proof above is independent of numerical optimization, symbolic-computation software, or unproved geometric classifications. The accompanying file \texttt{verify\_certificate.py} checks the rational arithmetic using Python's standard-library \texttt{Fraction} type and supplies additional rational interval enclosures using integer square roots. It also evaluates the original $168$ angle terms directly with high-precision decimal arithmetic and compares them with \cref{eq:scalar}. These direct evaluations are diagnostic checks; the strict counterexample follows from \cref{sec:certificate}.

The values in \cref{tab:values} illustrate why rounding obscures the improvement. The displayed decimals are numerical approximations; their signs are not used as premises in the proof.
\begin{table}[htbp]
 \centering
 \caption{Numerical values within the square-antiprism family.}
 \label{tab:values}
 \begin{tabular}{lll}
 \toprule
 Configuration & Height & Angular sum $\F(h)$\\
 \midrule
 Conjectured configuration & $0.978318343479$ & $79.501077891096$\\
 Explicit counterexample & $0.978000000000$ & $79.501079634630$\\
 Unique family maximizer & $0.977698583533$ & $79.501080174966$\\
 \bottomrule
 \end{tabular}
\end{table}

The exact scalar equation \cref{eq:root}, together with the monotonicity of $K$, is also a constructive specification of $\hstar$: bisection with rationally certified signs can enclose it to any requested accuracy. An algebraic characterization is provided in \cref{app:algebraic}; no choice among extraneous squared roots is needed to define the correct height.

%% file: sections/conclusion.tex
\section{Conclusion}

The explicit configuration $\A(489/500)$ has angular sum strictly greater than the radical value in Mouchet's exact eight-point conjecture. The conjectured configuration is not stationary with respect to its plane separation. Both assertions follow from a derivative factorization and short exact rational inequalities.

The square-antiprism family itself admits a unique optimal separation, characterized by \cref{eq:root}. Determining whether this corrected family optimum is globally optimal over all eight-point configurations remains outside the result proved here. The five-point, six-point, and asymptotic questions are also unaffected. A broader claim would require a global upper bound for arbitrary configurations, which this paper does not supply.

%% file: sections/disclosure.tex
\section{Disclosure}
\label{sec:disclosure}
The proof strategy and counterexample were produced by OpenAI’s GPT-5.6 Sol Ultra through Codex in
response to prompts from the author. Codex was also used to revise
the exposition and prepare the LaTeX manuscript. The author selected the
problem, directed the interactions and revisions, and is the sole named author. The AI system is acknowledged as a reasoning and writing tool, not
as an author. This disclosure is not a substitute for independent expert
mathematical review.

%% file: sections/appendix.tex
\section{The counterexample in explicit coordinates}
\label{app:coordinates}

The configuration used in \cref{thm:counterexample} consists of the following points. This table removes any dependence on evaluating trigonometric functions to construct the example.
\begin{center}
\begin{tabular}{cccc}
\toprule
Point & First coordinate & Second coordinate & Third coordinate\\
\midrule
$p_0$ & $1/\sqrt2$ & $0$ & $-489/1000$\\
$p_1$ & $0$ & $1/\sqrt2$ & $-489/1000$\\
$p_2$ & $-1/\sqrt2$ & $0$ & $-489/1000$\\
$p_3$ & $0$ & $-1/\sqrt2$ & $-489/1000$\\
$q_0$ & $1/2$ & $1/2$ & $489/1000$\\
$q_1$ & $-1/2$ & $1/2$ & $489/1000$\\
$q_2$ & $-1/2$ & $-1/2$ & $489/1000$\\
$q_3$ & $1/2$ & $-1/2$ & $489/1000$\\
\bottomrule
\end{tabular}
\end{center}
Within either layer the four points are distinct, and points from different layers have different third coordinates. Thus the configuration belongs to $\Q_8$. The horizontal side lengths are one, and the difference between the two third coordinates is $978/1000=489/500$, as required.

\section{An algebraic characterization of the family optimum}
\label{app:algebraic}

\begin{proposition}
\label{prop:quartic}
Let $s=\sqrt2$. The polynomial
\begin{equation}
 P(x)=x^4-2x^3+(13s-16)x^2-2sx+s
 \label{eq:quartic}
\end{equation}
has exactly one real zero $x>1$. If that zero is denoted by $x_*$, the unique optimal height of \cref{thm:family} is
\begin{equation}
 \hstar=\sqrt{x_*^2-1+s/2}.
 \label{eq:algebraic-height}
\end{equation}
\end{proposition}
\begin{proof}
Set $x=d_-(\hstar)$. Then $d_+(\hstar)^2=x^2+s$, and \cref{eq:root} becomes
\begin{equation}
 \frac1x+\frac{c}{\sqrt{x^2+s}}=1,
 \qquad c=3-2s>0.
 \label{eq:x-root}
\end{equation}
The second term is positive, so $1/x<1$ and $x>1$. Rearrangement of \cref{eq:x-root} gives
\[
 \sqrt{x^2+s}=\frac{cx}{x-1}.
\]
Both sides are positive. Squaring and multiplying by $(x-1)^2$ is therefore reversible for $x>1$ and gives
\begin{align*}
 0&=(x^2+s)(x-1)^2-c^2x^2\\
  &=x^4-2x^3+(1+s-c^2)x^2-2sx+s.
\end{align*}
Since $c^2=(3-2s)^2=17-12s$, we have $1+s-c^2=13s-16$. Thus $P(x)=0$.

Conversely, suppose $x>1$ and $P(x)=0$. The expansion just given shows
\[
 (x^2+s)(x-1)^2=c^2x^2.
\]
All factors whose square roots are to be taken are positive, so
\[
 \sqrt{x^2+s}\,(x-1)=cx.
\]
Dividing by the positive quantity $x\sqrt{x^2+s}$ gives \cref{eq:x-root}. Now define $h=\sqrt{x^2-1+s/2}$, which is real and positive because $x>1$. Then $d_-(h)=x$ and $d_+(h)=\sqrt{x^2+s}$, so \cref{eq:x-root} says $K(h)=0$. The uniqueness of this positive height was proved in \cref{thm:family}. It follows that $h=\hstar$, and then $x=d_-(\hstar)$ is unique as well. This proves all assertions.
\end{proof}

\section{Exact interval checks and their interpretation}
\label{app:intervals}

The supplied verifier uses only the Python standard library. Its basic certificate is the rational calculation in \cref{sec:certificate}. Its additional interval calculations yield the following enclosures, where every terminating decimal is interpreted as an exact rational number:
\begin{align*}
 0.97769858&<\hstar<0.97769859,\\
 79.501077891095953&\leq\F(\hzero)\leq79.501077891095954,\\
 79.501079634630014&\leq\F(h_1)\leq79.501079634630015,\\
 79.501080174965714&\leq\F(\hstar)\leq79.501080174965715,\\
 0.000001743534060&\leq\F(h_1)-\F(\hzero)\leq0.000001743534061.
\end{align*}
These supplementary finite computations are not used in the analytic proof. To specify their verification completely, we describe the arithmetic and the root-enclosure procedure.

Represent a real number by a closed rational interval containing it. Addition and subtraction use endpoint addition and subtraction; multiplication takes the minimum and maximum of the four endpoint products. For an interval $[a,b]$ not containing zero, its reciprocal is $[1/b,1/a]$. These rules follow from monotonicity of addition and reciprocal, and from the fact that a bilinear function on a rectangle attains its extrema at its corners. Rational arithmetic is exact throughout.

For a nonnegative rational $r=p/q$ with $p\geq0$ and $q>0$, and an integer scale $M=10^{50}$, let
\[
 n=\left\lfloor\sqrt{\left\lfloor pM^2/q\right\rfloor}\right\rfloor.
\]
The integer square-root routine computes this $n$ exactly. By definition,
\[
 n^2\leq\left\lfloor\frac{pM^2}{q}\right\rfloor<(n+1)^2.
\]
Because the middle quantity is an integer, it is at most $(n+1)^2-1$. Hence
\[
 n^2\leq\frac{pM^2}{q}<(n+1)^2,
\]
which proves
\[
 \frac{n}{M}\leq\sqrt{r}<\frac{n+1}{M}.
\]
For an interval $[a,b]\subset[0,\infty)$, apply this lower bound to $a$ and this upper bound to $b$. Monotonicity of square root proves containment of the entire image. In particular, repeated appearances of $\sqrt2$ need not be treated as independent exact numbers: the interval operations remain valid even though they may enlarge the enclosure.

Starting with $[97769858/10^8,97769859/10^8]$, the verifier establishes $K$ positive at the left endpoint and negative at the right endpoint by these interval rules. Continuity and strict decrease then enclose the unique root. Sixty bisection steps refine this bracket, with each retained half determined by a rigorously signed interval for $K$ at the midpoint. If an interval were to include zero, the verifier would stop instead of inferring a sign. Substitution of the resulting height interval into \cref{eq:a,eq:b,eq:scalar} encloses $\F(\hstar)$. The same formulas evaluated at $h_1$ and an interval for $\hzero$ give the other value enclosures. Subtracting the latter intervals gives the improvement enclosure above.

The independent decimal calculation serves a different purpose: it builds the eight points and directly enumerates the three vertex cosines of each of the $56$ unordered triples. It checks the geometric reduction and transcription of the source constant. Agreement of decimal calculations alone would not prove a strict inequality; that role is fulfilled by the rational certificate.





%% file: references.bib
@article{Mouchet2006JMP,
  author = {Mouchet, Amaury},
  title = {Upper and lower bounds for an eigenvalue associated with a positive eigenvector},
  journal = {Journal of Mathematical Physics},
  volume = {47},
  number = {2},
  pages = {022109},
  year = {2006},
  doi = {10.1063/1.2168124},
  url = {https://arxiv.org/abs/math/0505541},
  note = {Preprint first submitted May 25, 2005; Conjecture 3 ($C_8$), preprint p. 11, journal p. 022109-8}
}

@article{Mouchet2006NPA,
  author = {Mouchet, Amaury},
  title = {Bounding the ground-state energy of a many-body system with the differential method},
  journal = {Nuclear Physics A},
  volume = {765},
  pages = {319--341},
  year = {2006},
  doi = {10.1016/j.nuclphysa.2005.10.015},
  url = {https://arxiv.org/abs/nucl-th/0506069}
}

@article{HiriartUrruty2009,
  author = {Hiriart-Urruty, Jean-Baptiste},
  title = {A new series of conjectures and open questions in optimization and matrix analysis},
  journal = {ESAIM: Control, Optimisation and Calculus of Variations},
  volume = {15},
  number = {2},
  pages = {454--470},
  year = {2009},
  doi = {10.1051/cocv:2008040},
  url = {https://www.numdam.org/articles/10.1051/cocv:2008040/},
  note = {Problem 4, pp. 458--459; published online June 24, 2008}
}
